\documentclass[11pt]{article}

\usepackage{graphicx}
\usepackage{epsfig}
\usepackage{epsf}
\usepackage{latexsym}
\usepackage{bbm}
\usepackage{float}
\usepackage{fancyvrb}
\usepackage[small]{caption}
\usepackage{url}

\usepackage{amsthm}
\usepackage{amsfonts}
\usepackage{amssymb}
\usepackage{amsmath}
\usepackage{enumerate}

\newtheorem{theorem}{Theorem}
\newtheorem{lemma}{Lemma}

\newtheorem{observation}{Observation}

\newcommand{\eps}{\varepsilon}

\newcommand{\RR}{\mathbb{R}}

\def\A{{\mathcal A}}
\def\P{{\mathcal P}}
\def\Q{{\mathcal Q}}
\def\R{{\mathcal R}}
\def\V{{\mathcal V}}

\newcommand{\vol}{{\rm vol}}

\def\Prob{{\rm Prob}}

\def\etal{{et~al.}}
\def\ie{{i.e.}}
\def\eg{{e.g.}}

\newcommand{\later}[1]{}
\newcommand{\old}[1]{}

\title{\textsc{Perfect vector sets, properly overlapping partitions,
and largest empty box}}
\author{
Adrian Dumitrescu\thanks{
Department of Computer Science,
University of Wisconsin--Milwaukee, USA\@.
Email:~\texttt{dumitres@uwm.edu}}
\qquad
Minghui Jiang\thanks{
Department of Computer Science,
Utah State University, Logan, USA\@.
Email:~\texttt{mjiang@cc.usu.edu}}}

\begin{document}

\maketitle

\begin{abstract}
  We revisit the following problem (along with its higher dimensional variant):
Given a set $S$ of $n$ points inside an axis-parallel rectangle $U$ in the
plane, find a maximum-area axis-parallel sub-rectangle that is contained in $U$
but contains no points of $S$. 

(I) We present an algorithm that finds a large empty box 
amidst $n$ points in $[0,1]^d$: 
a box whose volume is at least  $\frac{\log{d}}{4(n + \log{d})}$
can be computed in $O(n+d \log{d})$ time.

(II) To better analyze the above approach, 
we introduce the concepts of perfect vector sets and properly overlapping
partitions, in connection to the minimum volume of a maximum empty box amidst
$n$ points in the unit hypercube $[0,1]^d$, and derive bounds on their sizes.


\end{abstract}

\section{Introduction} \label{sec:intro}

Given an axis-parallel rectangle $U$ in the plane containing $n$ points,
{\sc Maximum Empty Rectangle} is the problem of computing a maximum-area
axis-parallel empty sub-rectangle contained in $U$. This problem is one of the
oldest in computational geometry, with multiple applications, \eg, in facility
location problems~\cite{NLH84}. In higher dimensions, finding the largest empty
box has applications in data mining, such as finding large gaps in a
multi-dimensional data set~\cite{EGLM03}.   

A \emph{box} in $\RR^d$, $d \ge 2$, is an open axis-parallel hyperrectangle
$(a_1,b_1)\times\cdots\times(a_d,b_d)$ with $a_i < b_i$ for $1 \le i \le d$.
Due to the fact that the volume ratio of any box inside another box is
invariant under scaling, the problem can be reduced to the case when the
enclosing box is a hypercube. Given a set $S$ of $n$ points in the unit
hypercube $U_d=[0,1]^d$, $d \ge 2$, an \emph{empty box} is a box empty of
points in $S$ and contained in $U_d$, and {\sc Maximum Empty Box} is the
problem of finding an empty box with the \emph{maximum} volume. Note that an
empty box of maximum volume must be \emph{maximal} with respect to inclusion.
Some planar examples of maximal empty rectangles are shown in Fig.~\ref{f1}.
All rectangles and boxes considered in this paper are axis-parallel.  
\begin{figure}[htb]
\centering\includegraphics[scale=0.67]{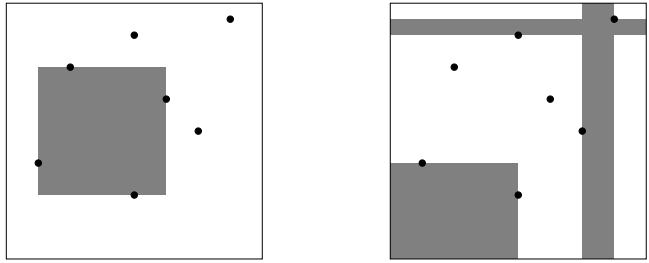}
\caption{A maximal empty rectangle supported by one point on each side
  (left), and three maximal empty rectangles supported by both points
  and sides of $[0,1]^2$ (right).}
\label{f1}
\end{figure}

According to an early result of Naamad, Lee, and Hsu~\cite{NLH84},
the number of maximal empty rectangles amidst $n$ points in the unit square
is $O(n^2)$ (and it is easy to exhibit tight examples);
as~such, the number of maximum empty rectangles
amidst $n$ points in the unit square is also $O(n^2)$.
Since then, this quadratic upper bound has been revisited
numerous times~\cite{AKM+87,AK90,AS87,AF86,CDL86,DS00,KMNS12,RT96}.
Only recently was the latter upper bound sharply reduced, to nearly linear, 
namely $O(n \log{n} \, 2^{\alpha(n)})$;
here $\alpha(n)$ is the extremely slowly growing inverse of Ackermann's
function\footnote{See \eg~\cite{SA95} for technical details on this and
other similar functions.}.
For any fixed $d \ge 2$,
the number of maximum empty boxes
amidst $n$ points in $U_d=[0,1]^d$, $d \ge 2$,
is always $O(n^d)$~\cite{KRSV08,DJ16a}
and sometimes $\Omega(n^{\lfloor d/2 \rfloor})$~\cite{DJ16a}.

Besides the number of maximum empty boxes, the volume of such boxes
is another parameter of interest.
Given a set $S$ of $n$ points in the unit hypercube $U_d=[0,1]^d$, 
where $d \geq 2$, let $A_d(S)$ be the maximum volume of an empty box
contained in $U_d$, and let $A_d(n)$ be the minimum value of $A_d(S)$
over all sets $S$ of $n$ points in $U_d$.
Rote and Tichy~\cite{RT96} proved that
$A_d(n)=\Theta\left(\frac{1}{n}\right)$ for any fixed $d \geq 2$.
From one direction, for any $d \ge 2$, we have
\begin{equation}\label{eq:upper}
A_d(n) < \left( 2^{d-1} \prod_{i=1}^{d-1} p_i \right) \cdot \frac{1}{n},
\end{equation}
where $p_i$ is the $i$th prime, as shown in~\cite{RT96,DJ13a}
using Halton-Hammersley generalizations~\cite{Hal60,Ham60}
of the van der Corput point set~\cite{C35a,C35b};
see also~\cite[Ch.~2.1]{Ma99}.

From the other direction,
by slicing the hypercube with $n$ parallel hyperplanes, each incident to one of
the $n$ points, the largest slice gives an empty box of volume at least
$\frac{1}{n+1}$,
and hence we have the lower bound $A_d(n) \geq \frac{1}{n+1}$ for each $d$. 
This trivial estimate can be improved using the following
inequality~\cite{DJ13a,DJ14} that relates $A_d(n)$ to $A_d(b)$
for fixed $d \ge 2$ and $b \ge 2$:
\begin{equation} \label{eq:adb}
A_d(n) \ge \big((b + 1) A_d(b) - o(1)\big) \cdot \frac1n.
\end{equation}
In particular, with $b=4$, the following bound\footnote{%
A weaker bound with $b=3$ was inadvertently labeled as an improvement
over this bound in~\cite{DJ14}.}
was obtained in~\cite{DJ13a}:
$$
A_d(n) \ge A_2(n) \ge \big(5 A_2(4) - o(1)\big) \cdot \frac1n
= \big(1.25 - o(1)\big) \cdot \frac1n.
$$

By exploiting the above observation of~\eqref{eq:adb}
in a more subtle and fruitful way, 
Aistleitner, Hinrichs, and Rudolf~\cite{AHR15} recently proved that
$A_d(\lfloor \log{d} \rfloor) =\Omega(1)$. It follows that the dependence
on $d$ in the volume bound is necessary, \ie, the maximum volume grows with the
dimension $d$.
As a consequence, the following lower bound is derived in~\cite{AHR15}:
\begin{equation} \label{eq:lower}
  A_d(n) \geq \frac{\log{d}}{4 (n + \log{d})}.
\end{equation}

Following this new development, we present an algorithm
that finds a large empty box amidst $n$ points in $[0,1]^d$,
whose volume is at least $\frac{\log{d}}{4(n + \log{d})}$,
in $O(n + d \log d)$ time.
Also, inspired by the technique of~\cite{AHR15},
we introduce the concepts of \emph{perfect vector sets}
and \emph{properly overlapping partitions} as tools for
bounding the minimum volume of a maximum empty box amidst $n$ points in the
unit hypercube $U_d=[0,1]^d$.
We show the equivalence of these two concepts,
then derive an exact closed formula for the maximum size of a family of
pairwise properly overlapping $2$-partitions of $[n]$, and obtain exponential
lower and upper bounds (in $n$) on the maximum size of a family of $t$-wise
properly overlapping $a$-partitions of $[n]$ for all $a \geq 2$ and $t \geq 2$.
These new concepts and corresponding bounds are connected to classical concepts
in extremal set theory such as Sperner systems and the LYM
inequality~\cite{Bo86}, and will likely see other applications. 

\paragraph{Notations.}
Let $[n]$ denote the set $\{1,2,\ldots,n\}$. 
For  $A \subset [n]$, $\overline{A} =[n] \setminus A$ denotes the complement of $A$. 
As usually, $\Theta, O, \Omega$ notation is used to describe the
asymptotic growth of functions.
When writing $f \sim g$, we ignore constant factors. 
The $\Omega^*$ notation is used to describe the asymptotic growth
of functions ignoring polynomial factors; if $1<c_1<c_2$ are two constants,
we frequently write $\Omega^*(c_2^n) = \Omega(c_1^n)$.

\section{A fast algorithm for finding a large empty box} \label{sec:algorithm}

We first give an efficient algorithm for finding a large empty box,
\ie, one whose volume is at least that guaranteed by equation~\eqref{eq:lower}. 
We essentially proceed as directed by the proof by Aistleitner~\etal~\cite{AHR15}.

\begin{theorem}\label{thm:large}
Given $n$ points in $[0,1]^d$, an empty box of volume
at least  $\frac{\log{d}}{4(n + \log{d})}$
can be computed in $O(n+d \log{d})$ time.
\end{theorem}
\begin{proof}
Let $\ell=\lfloor \log{d} \rfloor$, and $k= \lfloor n/(\ell+1) \rfloor$.
First partition the $n$ points in $U_d$ into $k+1$ boxes of equal volume
by using parallel hyperplanes orthogonal to first axis.
Select the box, say $B$, containing the fewest points, at most $\ell$;
we may assume that $B$ contains exactly $\ell$ points in its interior.  
We have
\begin{equation} \label{eq:1}
\vol(B) =\frac{1}{k+1} \geq 
\frac{\ell+1}{n+\ell+1} \geq \frac{\log{d}}{n+\log{d}}. 
\end{equation}
Clearly, $B$ can be found in $O(n)$ time by examining the first
coordinate of each point and using the integer floor function. 
Assume that $B=[a,b] \times [0,1]^{d-1} = \prod_{i=1}^d [a_i,b_i]$.

Second, encode the $\ell$ points in $B$ by $d$ binary vectors of
length $\ell$,
$\V=\{\mathbf{v_1},\ldots,\mathbf{v_d}\}$, one for each coordinate:
The $j$th bit of the $i$th vector, for $j=1,\ldots,\ell$,
is set to $0$ or $1$ depending on whether the $i$th coordinate
of the $j$th point is $\leq (a_i +b_i)/2$ or $> (a_i +b_i)/2$, respectively.
Clearly, there are at most $2^\ell$ distinct binary vectors of length $\ell$. 

If there is a zero-vector in $\V$, say, $\mathbf{v_i}$,
all points are contained in the box
$$ \prod_{k<i}[a_k,b_k] \times \left[a_i,\frac{a_i +b_i}{2}\right] \times \prod_{i<k}[a_k,b_k], $$
and so the complementary box of volume $\vol(B)/2$ is empty; the same
argument holds if one of the $d$ vectors in $\V$ has
all coordinates equal to $1$. 
If neither of these cases occurs, since $2^\ell -2<d$, 
then by the pigeonhole principle there is pair of equal vectors, 
say $\mathbf{v_i},\mathbf{v_j}$, with $i<j$: \ie,
$\mathbf{v_i}[r] =\mathbf{v_j}[r]$ for each $r \in [\ell]$. 
In particular, if $\alpha \in \{01,10\}$, then 
$\mathbf{v_i}[r] \, \mathbf{v_j}[r] \neq \alpha $, for each $r \in [\ell]$; 
we say that the binary combination (string) $\alpha$ is \emph{uncovered} 
by this pair of vectors.
By construction, an uncovered combination, say $01$, yields an empty ``quarter'' of $B$:
$$ \prod_{k<i}[a_k,b_k] \times \left[a_i,\frac{a_i +b_i}{2}\right] \times \prod_{i<k<j}[a_k,b_k]
\times \left[\frac{a_j +b_j}{2},b_j\right] \times \prod_{j<k}[a_k,b_k]. $$
Its volume is obviously $\vol(B)/4$, thus in all cases one finds an empty box of volume
$\vol(B)/4$. 

The $d$ binary vectors of size $\ell$ can be viewed as $d$ integers in
the range from $0$ to $d$. These can be assembled in time $O(d \ell)=O(d \log{d})$.
Finding a pair of duplicate vectors is easily done by sorting the $d$ integers, 
say, using radix sort in $O(d)$ time~\cite{CLRS09}; or by other method in time 
$O(d \log{d})$. Use the uncovered binary combination to output the corresponding 
empty box of $U_d$. By~\eqref{eq:1}, its volume is at least that guaranteed
by equation~\eqref{eq:lower}, as required.
The total running time is $O(n+d \log{d})$, as claimed. 
\end{proof}

\paragraph{Remark.}
Slightly improved parameters can be chosen according to the theory of
perfect vectors sets, \eg, by Theorem~\ref{thm:p(n)} in Section~\ref{sec:perfect},
however the effects in the outcome are negligible.

\section{Perfect vector sets and properly overlapping partitions}
\label{sec:perfect}

\paragraph{Perfect vector sets.}
Let $n \geq 2$ and $\Sigma=\{0,1\}$.
A set of binary vectors $\V=\{\mathbf{v_1},\ldots,\mathbf{v_k}\}$,
where $\mathbf{v_1},\ldots,\mathbf{v_k} \in \{0,1\}^n$ is called  \emph{perfect} if
(i)~$|\V| \geq 2$ and
(ii)~for every pair $(\mathbf{v_i},\mathbf{v_j})$, $1 \leq i<j \leq k$, 
and for every $\alpha \in \{0,1\}^2$, we have 
$\mathbf{v_i}[r] \, \mathbf{v_j}[r] =\alpha $, for some $r \in [n]$.
We refer to the latter condition as the \emph{covering condition} for
the pair $(\mathbf{v_i},\mathbf{v_j})$ and the binary string $\alpha$.
Since $|\Sigma^2|=4$, the covering condition requires $n \geq 4$.
For example, writing the elements in $\Sigma^2$ as the $4$ rows of a
$4 \times 2$ binary matrix yields a perfect set of $2$ binary vectors as
the columns of this matrix. This shows the existence of perfect vector
sets of length $4$;
and the existence of perfect vector sets of any higher length is implied.
A vector set that is not perfect is called \emph{imperfect}. 

\paragraph{Remarks.}
Observe that the covering condition above implies the seemingly stronger
covering condition: for every unordered pair $\{i,j\} \subset [k]$
and for every $\alpha \in \{0,1\}^2$, we have 
$\mathbf{v_i}[r] \, \mathbf{v_j}[r] =\alpha $, for some $r \in [n]$.

Further, observe that every perfect multiset is actually a set of vectors, \ie, no
duplicates may exist. Indeed, assume that two elements of the multiset
are the same vector:
$\mathbf{v_i}=\mathbf{v_j} =\mathbf{v}$ for some $i<j$;
then the required covering condition
fails for this ordered pair for both $\alpha =01$ and $\alpha =10$.
We have thus shown that the notion of perfect vector sets cannot be extended to 
multisets. 

\smallskip
Let $p(n)$ denote the maximum size of a perfect set of vectors of
length $n \geq 4$; by the above observations, $2 \leq p(n) \leq 2^n$. 
In Theorem~\ref{thm:p(n)} we give a finer estimate of $p(n)$, in particular,
it is shown that $p(n)= {n-1 \choose \lfloor n/2 \rfloor -1} = \Theta(2^n \, n^{-1/2})$.

\paragraph{$t$-wise perfect vector sets.}
We extend the above setup for larger alphabets and for multiple vectors as follows.
Let $\Sigma_a=\{0,1,\ldots,a-1\}$, where $a \geq 2$; let $t \geq 2$.
A set of vectors $\V=\{\mathbf{v_1},\ldots,\mathbf{v_k}\}$,
where $\mathbf{v_1},\ldots,\mathbf{v_k} \in \Sigma_a^n$ is called
$t$-wise \emph{perfect} with respect to $\Sigma_a$ if
(i)~$|\V| \geq t$ and
(ii)~for for every $t$-uple $(\mathbf{v_{i_1}},\ldots,\mathbf{v_{i_t}})$, 
where $1 \leq i_1 < i_2 < \ldots < i_t \leq k$,
and for every $\alpha \in  \Sigma_a^t$, we have 
$\mathbf{v_{i_1}}[r] \ldots \mathbf{v_{i_t}}[r] =\alpha $, for some $r \in [n]$.
We refer to the latter condition as the $t$-wise \emph{covering condition} for
the $t$-uple $(\mathbf{v_{i_1}} \ldots \mathbf{v_{i_t}})$ and the string $\alpha$,
where $|\alpha|=t$.
If there exists a $t$-wise perfect set of vectors of length $n$ over the
alphabet $\Sigma_a$, then we must have $n \ge a^t$.
As in the binary case,  writing the elements in $\Sigma_a^t$ as the $a^t$ rows of a
$a^t \times t$ matrix yields a $t$-wise perfect set of $t$ vectors over $\Sigma_a$
as the columns of this matrix. This shows the existence of perfect
vector sets of length $a^t$;
and the existence of $t$-wise perfect vector sets of any higher length is implied.
A vector set that is not $t$-wise perfect is called $t$-wise \emph{imperfect}. 
Throughout this paper we assume that $a$ and $t$ are fixed and $n$ tends to infinity. 

\paragraph{Remarks.}
Clearly, if $s \leq t$, a vector set that is $t$-wise \emph{perfect}
with respect to $\Sigma_a$
is also $s$-wise \emph{perfect} with respect to $\Sigma_a$. 
Again, the covering condition above implies the seemingly stronger
covering condition that takes $t$ vector indexes in any order. 
Finally, every perfect multiset is in fact a set of vectors, \ie, no
duplicates may exist; that is, 
the notion of $t$-wise perfect vector sets cannot be extended to multisets. 

\smallskip
Let $p(a,t,n)$ denote the maximum size of a $t$-wise perfect set of
vectors of length $n \geq a^t$ 
over $\Sigma_a$. By the above observations, $t \leq p(a,t,n) \leq a^n$.
By slightly abusing notation, we write $p(n)$ instead of $p(2,2,n)$. 

\paragraph{Properly overlapping partitions.}
For any $a \ge 2$ and $t \ge 2$,
we say that a family $\P$ of (unordered) $a$-partitions of a set
\emph{$t$-wise properly overlap} if
(i)~$|\P| \geq t$ and 
(ii)~for any subfamily of $t$ $a$-partitions $P_1, \ldots, P_t$ in $\P$,
the intersection of any $t$ parts, with one part from each $P_i$, is nonempty.
Observation~\ref{obs:equiv} below shows that $p(a, t, n)$,
from the earlier setup with perfect vector sets,
can be defined alternatively as the maximum size of a family of $t$-wise properly
overlapping $a$-partitions of $[n]$. We thus must have $n \geq a^t$.
\begin{observation} \label{obs:equiv}
  Any family of $t$-wise perfect set of vectors of length $n$
  over the alphabet $\Sigma_a$
  can be put into a one-to-one correspondence with a same-size family of $t$-wise
  properly overlapping $a$-partitions of $[n]$. 
  Conversely, any family of $t$-wise properly overlapping $a$-partitions of $[n]$
  can be put into a one-to-one correspondence with a same-size family of
  $t$-wise perfect set of vectors of length $n$ over the alphabet $\Sigma_a$.
\end{observation}
\begin{proof}
 Let $\V$ denote a family of $t$-wise perfect set of vectors of length $n$ over the
 alphabet $\Sigma_a$. Construct a  family of partitions of $[n]$ as follows:
 For any vector $\mathbf{v} \in \V$, consider the $a$-partition of $[n]$ in which
 element $r$ belongs to the set $\mathbf{v}[r]$, $r=1,2,\ldots,n$.
 One can see that the above correspondence is one-to-one.
 
  Suppose now that $\P$ is a family of $t$-wise properly overlapping
  $a$-partitions of $[n]$. 
  For any $a$-partition of $[n]$ consider the vector whose
  $r$th position is the number of the set containing $r$ (an element of $[a]$). 
  One can see that the above correspondence is one-to-one.

 Second, the $t$-wise perfect condition with respect to $\V$ is the same as the
 $t$-wise properly overlapping condition with respect to $\P$: indeed,
 the $t$-wise covering condition
 for the $t$-uple $(\mathbf{v_{i_1}} \ldots \mathbf{v_{i_t}})$ 
 and the string $\alpha$ is nothing else than the properly overlapping condition 
 for the corresponding $t$ $a$-partitions $P_{i_1},\ldots,P_{i_t}$, \ie, 
 the intersection of any $t$ parts, with one part from each $P_i$, is nonempty. 
\end{proof}

Note, if $s \leq t$, then any family of \emph{$t$-wise properly overlapping}
$a$-partitions of $[n]$ are also $s$-wise properly overlapping,
thus if $n \geq a^t$, then 
$p(a, t, n) \le p(a, s, n)$; in particular $p(a, t, n) \le p(a, 2, n)$.
Asymptotics of $p(a,2,n)$ for some small values of $a$, as implied by 
Theorems~\ref{thm:p(a,t,n)} and~\ref{thm:p(a,2,n)} are displayed in Table~\ref{table},
together with the exact value of $p(2,2,n)$ from Theorem~\ref{thm:p(n)}. 
The exact statements and the proofs are to follow.
\begin{table}[hbtp]
\begin{center}
\begin{tabular}{||c||c|c|c|c||} \hline
$a$ & 2 & 3 & 4 & 10 \\ \hline
  lower bd. on $p(a,2,n)$ & ${n-1 \choose \lfloor n/2 \rfloor -1}$ 
  & $\Omega(1.25^n)$ & $\Omega(1.12^n)$ & $\Omega(1.01^n)$ \\ \hline
  upper bd. on $p(a,2,n)$  & ${n-1 \choose \lfloor n/2 \rfloor -1}$
  & $O(1.89^n)$ & $O(1.76^n)$ &  $O(1.39^n)$ \\ \hline
\end{tabular}
\end{center}
\caption{$p(a,2,n)$ for a few small $a$.}
\label{table}
\end{table}

\section{An exact formula for $p(n) = p(2,2,n)$} \label{sec:p(n)}

In this section we prove the following exact formula:
\begin{theorem}\label{thm:p(n)}
For any $n \ge 4$, we have
\begin{equation} \label{eq:2}
  p(n) = {n-1 \choose \lfloor n/2 \rfloor -1}.
\end{equation}
\end{theorem}

\paragraph{Lower bound.}
Consider the family $\P$ consisting of all $2$-partitions of the form $A_i \cup B_i$,
where $1 \in A_i$, and $|A_i|=\lfloor n/2 \rfloor$.
We clearly have $|\P|= {n-1 \choose \lfloor n/2 \rfloor -1}$. So it only
remains to show that the $2$-partitions in $\P$ are properly overlapping. 
Let $i<j$. Since $1 \in A_i$ and $1 \in A_j$
it follows that $A_i \cap A_j \neq \emptyset$.
The same premise also implies that $B_i \cup B_j \subseteq \{2,3,\ldots,n\}$; since
$|B_i| = |B_j| =\lceil n/2 \rceil$, it follows that $B_i \cap B_j \neq \emptyset$.
We now show that $A_i \cap B_j \neq \emptyset$; assume for contradiction that
$A_i \cap B_j = \emptyset$;
since $|A_i|=\lfloor n/2 \rfloor$ and $|B_j| =\lceil n/2 \rceil$,
we have $B_j =\overline{A_i}$; however, $B_i =\overline{A_i}$; and so
$B_i=B_j$ and $A_i=A_j$; that is, $A_i \cup B_i = A_j \cup B_j$ is the same
$2$-partition, which is a contradiction.
We have shown that $A_i \cap B_j \neq \emptyset$; a symmetric argument shows that
$A_j \cap B_i \neq \emptyset$, hence the $2$-partitions in $\P$
are properly overlapping, as required. 

\paragraph{Upper bound.}
Consider a family $\P$ of properly overlapping $2$-partitions; write $|\P|=m$.
Each $2$-partition is of the form $A_i \cup B_i$, where (i)~$|A_i| \leq |B_i|$, and
(ii)~if $|A_i| = |B_i|$, then $1 \in A_i$.
Consider the family of sets $\A =\{A_1,\ldots,A_m\}$.
Since $\P$ consists of properly overlapping $2$-partitions,
$A_i \cap A_j \neq \emptyset$ for every $i \neq j$.

We next show that $A_i \not \subseteq A_j$,
for every $i \neq j$; that is, $\A$ is an \emph{antichain}.
In particular, this will imply that $\A$ consists of pairwise distinct
sets, \ie,  $A_i \neq A_j$ for every $i \neq j$. 
Assume for contradiction that $A_i \subseteq A_j$ for some $i \neq j$;
since $A_j \cap B_j =\emptyset$ we also have $A_i \cap B_j =\emptyset$,
contradicting the fact that the $2$-partitions in $\P$ are properly overlapping. 
We next show that $A_i \cup A_j \neq [n]$, for every $i \neq j$.
This holds if $1 \notin A_i$ and $1 \notin A_j$, since then $1 \notin A_i \cup A_j$.
It also holds if $1 \in A_i$ and $1 \in A_j$, since then $|A_i \cup A_j| \leq n-1$.
Assume now (for the remaining 3rd case) that $1 \in A_i$ and $1 \notin A_j$:
since $1 \notin A_j$, it follows that $A_j < n/2$,
and consequently, $|A_i \cup A_j| \leq n-1$.

To summarize, we have shown that $\A=\{A_1,\ldots,A_m\}$ consists of $m$ distinct sets
such that, if $i,j \in [m]$, $i \neq j$, then 
$$ A_i \cap A_j \neq \emptyset, \ \ \ \ A_i \not\subseteq A_j,
\ \ \ A_i \cup A_j \neq [n]. $$
It is known~\cite[Problem~6C, p.~46]{LW01} that under these conditions
$$ |\A| \leq {n-1 \choose \lfloor n/2 \rfloor -1}. $$
Since $|\A|=|\P|$, the same bound holds for $|\P|$ and this concludes the proof
of the upper bound on $p(n)$, and thereby the proof of Theorem~\ref{thm:p(n)}. 

\paragraph{Examples.}
By Theorem~\ref{thm:p(n)}, $p(4)=3$. 
$\V$ and $\P$ below correspond to each other and make a tight example:
$$ \V = \left\{
\begin{pmatrix}0\\0\\1\\1\end{pmatrix}, 
     \begin{pmatrix}0\\1\\0\\1\end{pmatrix}, 
       \begin{pmatrix}0\\1\\1\\0\end{pmatrix}
         \right\}. ~~~
\P= \big\{
         \{\{1,2\}, \{3,4\}\}, 
         \{\{1,3\}, \{2,4\}\}, 
         \{\{1,4\}, \{2,3\}\}
      \big\}. 
$$ 

By Theorem~\ref{thm:p(n)}, $p(5)=4$. 
 $\V$ and $\P$ below correspond to each other and make a tight example:
$$ \V = \left\{
\begin{pmatrix}0\\0\\1\\1\\1\end{pmatrix}, 
     \begin{pmatrix}0\\1\\0\\1\\1\end{pmatrix}, 
       \begin{pmatrix}0\\1\\1\\0\\1\end{pmatrix},
       \begin{pmatrix}0\\1\\1\\1\\0\end{pmatrix}
         \right\}. ~~~
$$
$$         \P= \big\{
         \{\{1,2\}, \{3,4,5\}\}, 
         \{\{1,3\}, \{2,4,5\}\}, 
         \{\{1,4\}, \{2,3,5\}\},
         \{\{1,5\}, \{2,3,4\}\}   
      \big\}. 
$$

\section{General bounds on $p(a,t,n)$} \label{sec:p(a,t,n)}

In this section we prove the following theorem:

\begin{theorem}\label{thm:p(a,t,n)}
  Let $a \geq 3$ and $t \geq 2$ be fixed. Then there exist constants
  $c_1=c_1(a,t)>0$,
  $\lambda_1=\lambda_1(a,t)>1$
$c_2=c_2(a)>0$,
  $\lambda_2=\lambda_2(a)<2$, 
and $n_0 (a,t) \geq a^t$ such that 
\begin{equation} \label{eq:3}
  p(a,t,n) \geq c_1 \lambda_1^n \text{ and }
  p(a,2,n) \leq c_2 \lambda_2^n,
\end{equation}
 for $n \geq n_0(a,t)$.
 In particular,
\begin{equation} \label{eq:4}
  p(a,t,n) \leq p(a,2,n) \leq {n-1 \choose \lfloor n/a \rfloor -1}.
\end{equation}
 \end{theorem}

\paragraph{Lower bound.}
To prove the lower bound on $p(a,t,n)$ in~\eqref{eq:3}
we construct a perfect set of vectors via a simple random construction. 
We randomly choose a set $\V=\{\mathbf{v_1},\ldots,\mathbf{v_k}\}$,
of $k \geq t$ vectors,
where each coordinate of each vector is chosen uniformly at random
from $\Sigma_a=\{0,1,\ldots,a-1\}$, for a suitable $k$.
We then show that for the chosen $k$, the set of vectors satisfies the required
covering condition for each $t$-uple of vectors with positive probability.

For any $\alpha \in \Sigma_a^t$,
$1 \leq i_1 < i_2 < \ldots < i_t \leq k$, 
and $r \in [n]$, we have
 $$ \Prob(\mathbf{v_{i_1}}[r] \ldots \mathbf{v_{i_t}}[r] \neq \alpha)=1- a^{-t}. $$
 Let $E(\alpha,i_1,\ldots,i_k)$ be the bad event that
 $\mathbf{v_{i_1}}[r] \ldots \mathbf{v_{i_t}}[r] \neq \alpha$ for
 each $r \in [n]$.

 Clearly,
 $$ \Prob(E(\alpha,i_1,\ldots,i_k)) \leq (1- a^{-t})^n. $$
 Let $F$ be the bad event that there exists $\alpha \in \Sigma_a^t$,
and a $t$-uple $1 \leq i_1 < i_2 < \ldots < i_t \leq k$, so that
$E(\alpha,i_1,\ldots,i_k)$ occurs.  
Clearly
 $$ \Prob(F) \leq a^t {k \choose t} (1- a^{-t})^n \leq (ak)^t (1- a^{-t})^n . $$
Set now $k \geq t$ as large as possible so that $ \Prob(F) < 1$, that is,
$$ k < \frac{1}{a} \left( \frac{a}{(a^t -1)^{1/t}} \right)^n,
\text{  for } n \geq n_0(a,t). $$
Since $ \Prob(F) <1$, by the basic probabilistic method (see, \eg, \cite{AS00}),
we conclude that the chosen set of vectors is $t$-wise perfect with
nonzero probability.
To satisfy the above inequality and thereby guarantee its existence,
we set (for a small $\eps>0$)
$$ c_1(a,t)=\frac{1}{a} -\eps, \text{ and }
\lambda_1(a,t) = \frac{a}{(a^t -1)^{1/t}} >1, $$
and thereby complete the proof of the lower bound.
Observe that for any fixed $t \geq 2$, the sequence
$$ x_m = \frac{m}{(m^t -1)^{1/t}}, \ \ \ m \geq 2, $$
is strictly decreasing, $x_2 \leq 2/\sqrt{3}$ and its limit is $1$. 

\paragraph{Upper bound.} 
To bound $p(a, 2, n)$ from above as in~\eqref{eq:3},
let $\P$ be a family of $a$-partitions of $[n]$
that pairwise properly overlap; write $|\P|=m$.
Each $a$-partition is of the form $A_i \cup B_i \cup \ldots $, for $i=1,\ldots,m$,
where $|A_i| \leq |B_i| \leq \ldots$. By this choice, 
$|A_i| \leq \lfloor n/a \rfloor $ for all $i \in [n]$. 
Consider the family of sets $\A =\{A_1,\ldots,A_m\}$.
Since $\P$ consists of properly overlapping $a$-partitions,
$A_i \cap A_j \neq \emptyset$ for every $i \neq j$.

We next show that $A_i \not \subseteq A_j$,
for every $i \neq j$; that is, $\A$ is an \emph{antichain}.
In particular, this will imply that $\A$ consists of pairwise distinct
sets, \ie,  $A_i \neq A_j$ for every $i \neq j$. 
Assume for contradiction that $A_i \subseteq A_j$ for some $i \neq j$;
since $A_j \cap B_j =\emptyset$ we also have $A_i \cap B_j =\emptyset$,
contradicting the fact that the $a$-partitions in $\P$ are properly overlapping. 

To summarize, we have shown that $\A=\{A_1,\ldots,A_m\}$ consists of $m$ distinct sets
such that, if $i,j \in [m]$, $i \neq j$, then 
$$ A_i \cap A_j \neq \emptyset, \ \ \ \ A_i \not\subseteq A_j, $$
and $|A_i| \leq \lfloor n/a \rfloor $ for all $i \in [n]$. 
It is known~\cite[Theorem~6.5, p.~46]{LW01} that under these conditions
$$ |\A| \leq {n-1 \choose \lfloor n/a \rfloor -1}. $$
Since $|\A|=|\P|$, the same bound holds for $|\P|$.

By Stirling's formula,
$$ n! = \sqrt{2\pi n} \left( \frac{n}{e}\right)^n
\left( 1 + O\left(\frac{1}{n} \right) \right), $$
hence
\begin{equation} \label{eq:stirling}
{n-1 \choose \lfloor n/a \rfloor -1} \leq {n \choose \lfloor n/a \rfloor }
\sim \frac{1}{\sqrt{n}} \left( \frac{a}{((a-1)^{a-1})^{1/a}} \right)^n.
\end{equation}

Note that the sequence
$$ y_m = \frac{m}{((m-1)^{m-1})^{1/m}}, \ \ \ m \geq 2, $$
is strictly decreasing, $y_2=2$, and its limit is $1$.
By~\eqref{eq:stirling} we can therefore set
$$ c_2(a)>0 \text{ and } \lambda_2(a) = \frac{a}{((a-1)^{a-1})^{1/a}} < 2, $$
and note that if $a$ is sufficiently large, then $ \lambda_2(a,2) $
is arbitrarily close to $1$, in agreement with the behavior of
$\lambda_1(a,t)$, for large $a$; that is, for any fixed $t \geq 2$, we have 
$\lim_{m \to \infty} x_m = \lim_{m \to \infty} y_m =1$.

\section{Sharper bounds on $p(a,t,n)$} \label{sec:sharper}

We next derive sharper bounds for $t=2$ (in Theorem~\ref{thm:p(a,2,n)})
via an explicit lower bound construction and via an upper bound argument
specific to this case.

\begin{theorem}\label{thm:p(a,2,n)}
Let $b = {a \choose 2}$ and $k = \lfloor n/(2b) \rfloor$.
  Then the following inequalities hold:
\begin{equation} \label{eq:5}
\frac12 {2k \choose k}
\leq p(a,2,n) \leq {n\choose \lceil n/a \rceil} \Big{/} (2b). 
\end{equation}
\end{theorem}

\paragraph{An explicit lower bound.}
Let $b = {a \choose 2}$. Let $k = \lfloor n/(2b) \rfloor$. Then the
set $[n]$ can be partitioned into $b + 1$ subsets, including $b$ subsets
$B_{i j}$ of size $2k$, $1 \le i < j \le a$, and a possibly empty leftover
subset $C$.

Note that each $B_{ij}$ has size $2k$ and hence exactly $\frac12 {2k \choose k}$
$2$-partitions into two subsets of equal size $k$. We next construct a family
of $\frac12 {2k \choose k}$ pairwise properly overlapping $a$-partitions of $[n]$.

To obtain an $a$-partition $(P_1,\ldots,P_a)$, initialize each $P_i$ to an
empty set, then take a distinct $2$-partition of each $B_{ij}$ and put the
elements of the two parts into $P_i$ and $P_j$, respectively. Then each $P_i$
has size $k(a-1)$. Finally, if the leftover subset $C$ is not empty, add its
elements to $P_1$.

For any two $a$-partitions $(P_1,\ldots,P_a)$ and $(Q_1,\ldots,Q_a)$ thus
constructed, and for any pair $i<j$, 
the intersection of any one of $P_i,P_j$ and any one of $Q_i,Q_j$
is not empty because in each case, the two sets contain two distinct 
non-complementary $k$-subsets of the same $2k$-set $B_{ij}$.
Hence these $a$-partitions are pairwise properly overlapping as desired.

Finally, note that the size of this family is $\frac12 {2k \choose k}$, which is
about $2^{n/b}$, ignoring polynomial factors. When $a = 3$,
$b = {3 \choose 2} = 3$, we have a lower bound 
$p(3,2,n)= \Omega^*\big( (2^{1/3})^n \big) = \Omega(1.25^n)$.

We illustrate the construction for $a=3$, $n=12$; we get $k=2$, $|B_{ij}|=4$,
for $1 \leq i<j \leq 3$; and
$B_{12}=\{1,2,3,4\}$, $B_{13}=\{5,6,7,8\}$,  $B_{23}=\{9,10,11,12\}$.
Each $B_{ij}$ has three $2$-partitions; denote by $\P_{ij}$ the corresponding family.
\begin{align*}
\P_{12} &=\big\{ 
         \{\{1,2\}, \{3,4\}\}, 
         \{\{1,3\}, \{2,4\}\}, 
         \{\{1,4\}, \{2,3\}\}
      \big\}, \\ 
\P_{13} &=\big\{ 
         \{\{5,6\}, \{7,8\}\}, 
         \{\{5,7\}, \{6,8\}\}, 
         \{\{5,8\}, \{6,7\}\}
      \big\}, \\
\P_{23} &=\big\{ 
         \{\{9,10\}, \{11,12\}\}, 
         \{\{9,11\}, \{10,12\}\}, 
         \{\{9,12\}, \{10,11\}\}
      \big\}. 
\end{align*}
The resulting three $3$-partitions are:
\begin{align*}
\P&=\big\{ \{1,2,5,6\}, \{3,4,9,10\}, \{7,8,11,12\} \big\}, \\
\Q&=\big\{ \{1,3,5,7\}, \{2,4,9,11\}, \{6,8,10,12\} \big\}, \\
\R&=\big\{ \{1,4,5,8\}, \{2,3,9,12\}, \{6,7,10,11\} \big\}.
\end{align*}

\bigskip
For the upper bound we need the following two technical lemmas.

\begin{lemma}\label{lem:sum}
Let $a \ge 2$,
$n_i \ge 1$ for $1 \le i \le a$, and $n = \sum_{i=1}^a n_i$.
Then
$$
\sum_{i=1}^a \frac{1}{{n\choose n_i}}
\ge \frac{a}{{n\choose \lceil n/a \rceil}}.
$$
\end{lemma}

\begin{proof}
The lemma clearly holds for $a = 2$
since ${n\choose n_i}$ is maximized at $n_i = \lfloor n/2 \rfloor$
or $\lceil n/2 \rceil$.
Now let $a \ge 3$.
First observe that we can have $n_i > \lfloor n/2 \rfloor$ for at most one
$n_i$.
If $n_i > \lfloor n/2 \rfloor$ for some $n_i$,
then we must have $n_j < \lfloor n/2 \rfloor$ for some $n_j$.
But then
$1/{n\choose n_i} \ge 1/{n\choose n_i-1}$
and
$1/{n\choose n_j} \ge 1/{n\choose n_j+1}$,
where $n_i-1$ is less than $n_i$,
and $n_j+1$ remains at most $\lfloor n/2 \rfloor$.
Thus we can assume without loss of generality that
$n_i \le \lfloor n/2 \rfloor$ for all $n_i$.
Recall the extension of the factorial function $k!$ for integers $k$ to
the gamma function $\Gamma(x)$ for real numbers $x$,
where $\Gamma(k+1) = k!$.
Correspondingly, we can extend $1/{n \choose k}$
to a real function $f(x) = \Gamma(x+1)\Gamma(n-x+1)/\Gamma(n+1)$
such that $f(k) = 1/{n \choose k}$.
Since $f(x)$ is convex and decreasing for
$1 \le x \le \lfloor n/2 \rfloor$,
it follows by Jensen's inequality that
\begin{equation*}
\sum_{i=1}^a \frac{1}{{n\choose n_i}}
\ge a\cdot f(n/a)
\ge a\cdot f(\lceil n/a \rceil)
= \frac{a}{{n\choose \lceil n/a \rceil}}.
\tag*{\qedhere}
\end{equation*}
\end{proof}

\begin{lemma}\label{lem:pack}
Let $m \ge 2$, $n \ge 2$, and $b \ge 1$.
Let $\A = \{A_1,\ldots,A_m\}$ be a family of $m$ distinct subsets of $[n]$
such that
$|A_i \setminus A_j| \ge b$
and
$|A_j \setminus A_i| \ge b$
for any two subsets $A_i$ and $A_j$ in $\A$.
Then
$$
\sum_{i=1}^m \frac{b}{{n\choose |A_i|}}
\le 1.
$$
\end{lemma}

\begin{proof}
Our proof is an adaptation of the proof of~\cite[Theorem~6.6]{LW01}.
Let $\pi$ be a permutation of $[n]$ placed on a circle and let us say that
$A_i \in \pi$ if the elements of $A_i$ occur consecutively somewhere on that
circle.
Then each subset $A_i \in \pi$ corresponds to a closed circular arc with
endpoints in $[n]$.
For any two subsets $A_i$ and $A_j$ in $\pi$,
the condition
$|A_i \setminus A_j| \ge b$
and
$|A_j \setminus A_i| \ge b$
requires that the left (respectively, right) endpoints of the corresponding two
circular arcs on the circle differ by at least $b$ modulo $n$.
Therefore, if $A_i \in \pi$, then $A_j \in \pi$ for at most
$\lfloor n/b \rfloor$ values of $j$ including $i$.

Now define $f(\pi,i) = \frac{1}{\lfloor n/b \rfloor}$ if $A_i \in \pi$,
and $f(\pi,i) = 0$ otherwise.
By the argument above, we have $\sum_{\pi} \sum_{i=1}^m f(\pi, i) \le n!$.
Following a different order to evaluate the double summation, we can count,
for each fixed $A_i$, and for each fixed circular arc of $|A_i|$ consecutive
elements out of $n$ elements on the circle,
the number of permutations $\pi$ such that $A_i$ corresponds to the circular
arc,
which is exactly $|A_i|!(n - |A_i|)!$.
So we have
$$
\sum_{i=1}^m n\cdot |A_i|!(n - |A_i|)! \cdot \frac{1}{\lfloor n/b \rfloor}
\le n!,
$$
which yields the result.
\end{proof}

\paragraph{Upper bound.}
We now proceed to prove the upper bound in Theorem~\ref{thm:p(a,2,n)}.
Let $\P$ be a family of $a$-partitions of $[n]$ that pairwise properly overlap.
Then each part of any $a$-partition in $\P$ must have at least $a$ elements
to intersect the $a$ disjoint parts of any other $a$-partition in $\P$.
Thus for any two parts $A_i$ and $A_j$ of the same $a$-partition,
$|A_i \setminus A_j| = |A_i| \ge a$
and
$|A_j \setminus A_i| = |A_j| \ge a$.
On the other hand,
for any two parts $A_i$ and $A_j$ of two different $a$-partitions,
we must have
$|A_i \setminus A_j| \ge a-1$
so that $A_i$ can intersect the other $a-1$ parts of the $a$-partition that
includes $A_j$,
and symmetrically,
$|A_j \setminus A_i| \ge a-1$.
Thus the family of subsets in all $a$-partitions in $\P$ satisfies the
condition of Lemma~\ref{lem:pack} with $b = a-1$.
It follows that
$$
\sum_{\A\in\P}\sum_{A_i\in\A} \frac{a-1}{{n\choose |A_i|}} \le 1.
$$
Then, by Lemma~\ref{lem:sum}, we have
$$
|\P|\cdot \frac{a(a-1)}{{n\choose \lceil n/a \rceil}}
\le
\sum_{\A\in\P}\sum_{A_i\in\A} \frac{a-1}{{n\choose |A_i|}} \le 1.
$$
Thus the size of $\P$ is at most
${n\choose \lceil n/a \rceil}/(a(a-1))$.
Note that this upper bound matches our upper bound of
${n-1\choose \lfloor n/2 \rfloor - 1}$ when $a = 2$ and $n$ is even,
and improves the upper bound of 
${n-1\choose \lfloor n/a \rfloor - 1}$
by a factor of $\frac1{a-1}$ when $n$ is a multiple of $a$.

\section{Connections to classical concepts in extremal set theory}

A family $\A$ of sets is an \emph{antichain}
if for any two sets $U$ and $V$ in $\A$,
neither $U \subseteq V$ nor $V \subseteq U$ holds.
For $l \ge 1$,
a sequence $\langle T_0,T_1,\ldots,T_l \rangle$ of $l + 1$ sets
is an \emph{$l$-chain} (a chain of length $l$) if
$T_0 \subset T_1 \subset \ldots \subset T_l$.
A family of sets is said to be \emph{$r$-chain-free}
if it contains no chain of length $r$;
in particular, every antichain is $1$-chain-free.
Sperner~\cite{Sp28} bounded the largest size of an antichain $\A$
consisting of subsets of $[n]$:
$$
|\A| \le {n \choose \lfloor n/2 \rfloor},
$$
where equality is attained, for example, when $\A$ is the family of all
subsets of $[n]$ with exactly $\lfloor n/2 \rfloor$ elements.
Bollob\'as~\cite{Bo65}, Lubell~\cite{Lu66}, Yamamoto~\cite{Ya54},
and Meshalkin~\cite{Me63} independently discovered a stronger result
known as the LYM inequality:
$$
\sum_{A\in\A} \frac1{{n \choose |A|}} \le 1.
$$

For $p \ge 2$,
a \emph{$p$-composition} of a finite set $S$
is an ordered $p$-partition of $S$, that is,
a tuple $(A_1,\ldots,A_p)$ of $p$ disjoint sets whose union is $S$.
For any family $\A$ of $p$-compositions $A = (A_1,\ldots,A_p)$ of $[n]$,
the $i$th component of $\A$, $1 \le i \le p$,
is the family $\A_i := \{ A_i \mid A \in \A \}$
of subsets of $[n]$.
Meshalkin~\cite{Me63} proved that if each component $\A_i$, $1 \le i \le p$,
is an antichain, then the maximum size of a family $\A$ of $p$-compositions
is the largest $p$-multinomial coefficient
$$
|\A| \le {n \choose n_1,\ldots,n_p},
$$
where the $p$ integers $n_i$ sum up to $n$,
and any two of them differ by at most $1$.
Beck and Zaslavsky~\cite{BZ02} subsequently obtained an equality
on componentwise-$r$-chain-free families of $p$-compositions,
which subsumes the Meshalkin bound (as the $r=1$ case)
and generalizes the LYM inequality:
$$
\sum_{(A_1,\ldots,A_p)\in\A} \frac1{{n \choose |A_1|,\ldots,|A_p|}} \le r^{p-1}.
$$

Our concept of $t$-wise properly overlapping $a$-partitions
is analogous to the classical concept of componentwise-$r$-chain-free
$p$-compositions
when $t = 2$, $r = 1$, and $a = p$.
The difference in this case is that we consider unordered partitions and
require that all parts of all partitions pairwise overlap
and hence form an antichain
(as shown in the proof of Theorem~\ref{thm:p(n)}),
whereas Meshalkin~\cite{Me63} considers ordered partitions and
requires that in each component the corresponding parts
of all partitions form an antichain.

\paragraph{Added note.} 
After completion of the work on this manuscript, we learned that 
some of our results have been obtained earlier, in the the so-called framework of
``qualitative independent sets and partitions''.
More precisely, our properly $t$-wise overlapping partitions
have been sometimes referred to as qualitative $t$-independent partitions
or simply $t$-independent partitions in prior work.
For instance, it is worth pointing out that our Theorem~\ref{thm:p(n)}
was independently discovered by four papers with different
motivations~\cite{Bo73,BD72,Ka73,KS73}; see also~\cite{GKV92,GKV93,Ka91,KS92,PT89}
for other related results.
We also note that: (i)~the lower bound in~\cite[Theorem~4]{PT89} is a special case of
the explicit lower bound in our Theorem~\ref{thm:p(a,2,n)};
(ii)~the lower bound in~\cite[Theorem~5]{PT89} is analogous (and also obtained by
a probabilistic argument) to the lower bound in our Theorem~\ref{thm:p(a,t,n)}.
While some of our bounds are superseded by bounds in earlier papers
(\eg, the upper bound in~\cite[Theorem~1]{PT89} is stronger than the upper bounds in our
Theorems~\ref{thm:p(a,t,n)} and~\ref{thm:p(a,2,n)}), overall our results cover a broad landscape;
% cover a large area and paint a broad landscape;
as such, the writing has been left unaltered.  
Our main focus has been determining the asymptotic growth rate of $p(a,t,n)$
for fixed $a$ and $t$; Theorems~\ref{thm:p(n)},~\ref{thm:p(a,t,n)}, and~\ref{thm:p(a,2,n)}
provide the answers we need; their implications and connections with the maximum empty box
problem are discussed in the next section. %Section~\ref{sec:remarks}.

\section{Connections to maximum empty box and concluding remarks} \label{sec:remarks}

Our motivation for studying perfect vector sets and properly overlapping partitions
was determining whether the growth rate of $p(a,t,n)$ is exponential in $n$,
and its relation to the growth rate of $A_d(n)$ as a function in $d$. 
We next show within our framework of perfect vectors sets
(or that of properly overlapping partitions)
that a subexponential growth in $n$ of $p(a,t,n)$ would imply 
a superlogarithmic growth in $d$ of the maximum volume $A_d(n)$ via an argument
similar to that employed in the proof of Theorem~\ref{thm:large}; see also~\cite{AHR15}.

In the proof of Theorem~\ref{thm:large}, we have set $\ell=\lfloor \log{d} \rfloor$
and found a box $B$ containing exactly $\ell$ points in its interior
and with $\vol(B) \geq \frac{\ell+1}{n+\ell+1}$. We then encoded the $\ell$ points in $B$
by $d$ binary vectors of length $\ell$, $\V=\{\mathbf{v_1},\ldots,\mathbf{v_d}\}$. 
If $\V$ is perfect, we have $p(n) \leq 2^{n-1}$ by Theorem~\ref{thm:p(n)} if $n \geq 4$;
when applied to $\V$, this yields $d \leq 2^{\ell-1}$ and further that $\ell \geq \log{d} + 1$,
which is a contradiction. Thus $\V$ is imperfect, in which case 
an uncovered binary combination yields an empty box of volume $\vol(B)/4$ and we are done.

Similarly, assume for example that $p(a,t,n) < n^c$, for some $a,t \geq 2$, and a positive constant
$c>1$.  Set $\ell = \lfloor d^{1/c} \rfloor$ and proceed as above to find a box $B$
containing exactly $\ell$ points in its interior and with $\vol(B) \geq \frac{\ell+1}{n+\ell+1}$.
Encode the $\ell$ points in $B$ by $d$ vectors of length $\ell$ over $\Sigma_a=\{0,1,\ldots,a-1\}$
using the coordinates of the points and a uniform subdivision in $a$ parts of each 
extent of $B$; let $\V=\{\mathbf{v_1},\ldots,\mathbf{v_d}\}$. 
The $j$th bit of the $i$th vector, for $j=1,\ldots,\ell$,
is set to $k \in \{0,1,\ldots,a-1\}$ depending on whether the $i$th coordinate of the $j$th point
lies in the $(k+1)$th subinterval of the $i$th extent.
If $\V$ is perfect, since $p(a,t,n) < n^c$ by the assumption, 
this implies $d < \ell^c$, or $\ell > d^{1/c}$, which is a contradiction.
It follows that $\V$ is imperfect, in which case 
an uncovered $t$-wise combination yields an empty box of volume
$a^{-t} \, \vol(B) \geq a^{-t}  d^{1/c} /n$ and we are done.

By Theorem~\ref{thm:p(a,t,n)}, the growth rate of $p(a,t,n)$ is exponential in $n$,
and so the above scenario does not materialize.
This may suggest that $A_d(n)$ is closer to $\Theta \left(\frac{\log{d}}{n}\right)$
than to the upper bound in~\eqref{eq:upper} which is exponential in $d$.
In particular, it would be interesting to establish whether $A_d(n) \leq d^{O(1)} /n$. 

Recall that we have $A_d(\lfloor \log{d} \rfloor) =\Omega(1)$, as proved by 
Aistleitner~\etal~\cite{AHR15}; this gives a partial answer in relation to one of our
earlier open problems from~\cite{DJ13a}, namely whether $A_d(d) =\Omega(1)$;
this latter problem remains open. Under any circumstances, 
determining the asymptotic behavior of $A_d(n)$ remains an exciting open problem.

\paragraph{Acknowledgment.} We are grateful to Gyula Katona for bringing several articles
on qualitative independent sets and partitions to our attention.


\begin{thebibliography}{99}
%\itemsep 2pt

\bibitem{AKM+87}
A. Aggarwal, M. Klawe, S. Moran, P. Shor and R. Wilber,
Geometric applications of a matrix-searching algorithm,
\emph{Algorithmica}
\textbf{2} (1987), 195--208.   

\bibitem{AK90}
A. Aggarwal and M. Klawe,
Applications of generalized matrix searching to geometric algorithms,
\emph{Discrete Appl. Math.}
%\emph{Discrete Applied Mathematics}
\textbf{27} (1990), 3--23.

\bibitem{AS87}
A. Aggarwal and S. Suri,
Fast algorithms for computing the largest empty rectangle,
in: \emph{Proc. 3rd Ann. Sympos. on Comput. Geom.},
1987, pp.~278--290.

\bibitem{AHR15}
C. Aistleitner, A. Hinrichs, and D. Rudolf,
On the size of the largest empty box amidst a point set,
preprint, 2015, \verb+http://arxiv.org/abs/1507.02067v1+.

\bibitem{AS00}
N.~Alon and J.~Spencer,
\emph{The Probabilistic Method},
second edition, Wiley, New York, 2000.

\bibitem{AF86}
M. J. Atallah and G. N. Frederickson,
A note on finding a maximum empty rectangle,
\emph{Discrete Appl. Math.}
%\emph{Discrete Applied Mathematics}
\textbf{13(1)} (1986), 87--91.

\bibitem{BZ02}
M. Beck and T. Zaslavsky,
A shorter, simpler, stronger proof of the Meshalkin-Hochberg-Hirsch bounds on
componentwise antichains,
\emph{Journal of Combinatorial Theory, Series A},
\textbf{100} (2002), 196--199.

\bibitem{Bo65}
B. Bollob\'as,
On generalized graphs,
\emph{Acta Mathematica Academiae Scientiarum Hungaricae}
\textbf{16} (1965), 447--452.

\bibitem{Bo73}
B. Bollob\'as,
Sperner systems consisting of pairs of complementary subsets,
\emph{Journal of Combinatorial Theory, Series A},
\textbf{15} (1973), 363--366.

\bibitem{Bo86}
B. Bollob\'as,
Combinatorics: Set Systems, Hypergraphs, Families of Vectors, and
Combinatorial Probability,
Cambridge University Press, 1986. 

\bibitem{BD72}
B. A. Brace and D. E. Daykin,
Sperner type theorems for finite sets,
\emph{Combinatorics}
(Proc. Conf. Combinatorial Math., Oxford, 1972),
Inst. Math. Appl., Southend-on-Sea, 1972, 18--37.

\bibitem{CDL86}
B. Chazelle, R. Drysdale and D. T. Lee,
Computing the largest empty rectangle, 
%\emph{SIAM J. Comput.}
\emph{SIAM Journal on Computing}
\textbf{15} (1986), 300--315.

\bibitem{CLRS09} T.~H.~Cormen, C.~E.~Leiserson, R.~L.~Rivest, and C.~Stein,
\emph{Introduction to Algorithms}, 
3rd~edition, MIT Press, Cambridge, 2009.

\bibitem{C35a}
J. G. van der Corput,
Verteilungsfunktionen I.,
\emph{Proc.\ Nederl.\ Akad.\ Wetensch.}, 
\textbf{38} (1935), 813--821.

\bibitem{C35b}
J. G. van der Corput,
Verteilungsfunktionen II.,
\emph{Proc.\ Nederl.\ Akad.\ Wetensch.}, 
\textbf{38} (1935), 1058--1066.

\bibitem{DS00}
A. Datta and S. Soundaralakshmi,
An efficient algorithm for computing the
maximum empty rectangle in three dimensions,
\emph{Inform. Sci.}
%\emph{Information Sciences}
\textbf{128} (2000), 43--65.

\bibitem{DJ13a}
A. Dumitrescu and M. Jiang,
On the largest empty axis-parallel box amidst $n$ points,
{\em Algorithmica}
\textbf{66(2)} (2013), 225--248.   

\bibitem{DJ14}
A. Dumitrescu and M. Jiang,
Computational Geometry Column 60,
\emph{SIGACT News Bulletin}
\textbf{45(4)}, 2014, 76--82.

\bibitem{DJ16a}
A. Dumitrescu and M. Jiang,
On the number of maximum empty boxes amidst $n$ points,
{\em Proc. 32nd Ann. Sympos. Comput. Geometry}, June 2016, 
Leibniz International Proceedings in Informatics (LIPIcs) series, 
Schloss Dagstuhl; DOI: 10.4230/LIPIcs.SoCG.2016.36.

\bibitem{EGLM03}
J. Edmonds, J. Gryz, D. Liang, and R. Miller,
Mining for empty spaces in large data sets,
\emph{Theoret. Comp. Sci.}
%\emph{Theoretical Computer Science}
\textbf{296(3)} (2003), 435--452.

\bibitem{GKV92}
L. Gargano, J. K\"orner, and U. Vaccaro,
Qualitative independence and Sperner problems for directed graphs,
\emph{Journal of Combinatorial Theory, Series A}, 
\textbf{61} (1992), 173--192.

\bibitem{GKV93}
L. Gargano, J. K\"orner, and U. Vaccaro,
Sperner capacities,
\emph{Graphs and Combinatorics}, 
\textbf{9(1)} (1993), 31--46.

\bibitem{Hal60}
J. H. Halton,
On the efficiency of certain quasi-random sequences of points in
evaluating multi-dimensional integrals, 
\emph{Numer.\ Math.}
\textbf{2} (1960), 84--90.

\bibitem{Ham60}
J. M. Hammersley,
Monte Carlo methods for solving multivariable problems,
\emph{Ann.\ New York Acad.\ Sci.}
\textbf{86} (1960), 844--874.

\bibitem{KMNS12}
H. Kaplan, S. Mozes, Y. Nussbaum, and M. Sharir,
Submatrix maximum queries in Monge matrices and Monge partial matrices,
and their applications,
in \emph{Proc. 23rd ACM-SIAM Sympos. on Discrete Algorithms},
2012, pp.~338--355.

\bibitem{KRSV08}
H. Kaplan, N. Rubin, M. Sharir, and E. Verbin,
Efficient colored orthogonal range counting,
\emph{SIAM J. Comput.}
%\emph{SIAM Journal on Computing}
\textbf{38} (2008), 982--1011.

\bibitem{Ka73}
G. O. H. Katona,
Two applications (for search theory and truth functions) of Sperner type theorems,
\emph{Periodica Mathematica Hungarica}
\textbf{3(1-2)} (1973), 19--26.

\bibitem{Ka91}
G. O. H. Katona,
R\'enyi and the combinatorial search problems,
\emph{Studia Scientiarum Mathematicarum Hungarica}
\textbf{26} (1991), 363--378.

\bibitem{KS73}
D. Kleitman and J. Spencer,
Families of $k$-independent sets,
\emph{Discrete Mathematics} \textbf{6(3)} (1973), 255--262.

\bibitem{KS92}
J. K\"orner and G. Simonyi,
A Sperner-type theorem and qualitative independence,
\emph{Journal of Combinatorial Theory, Series A}, 
\textbf{59} (1992), 90--103.

\bibitem{LW01}
J. H. van Lint and R. M. Wilson,
\emph{A Course in Combinatorics},
Cambridge University Press, 2nd edition, New York, 2001.

\bibitem{Lu66}
D. Lubell,
A short proof of Sperner's lemma,
\emph{Journal of Combinatorial Theory}
\textbf{1} (1966), 299.

\bibitem{Ma99}
J. Matou\v{s}ek,
\emph{Geometric Discrepancy: An Illustrated Guide},
Springer, 1999.

\bibitem{Me63}
L. D. Meshalkin,
Generalization of Sperner's theorem on the number of subsets of a finite set,
\emph{Theory of Probability and its Applications},
\textbf{8} (1963), 203--204.

\bibitem{NLH84}
A. Namaad, D. T. Lee, and  W.-L. Hsu,
On the maximum empty rectangle problem,
\emph{Discrete Appl. Math.}
%\emph{Discrete Applied Mathematics}
\textbf{8} (1984), 267--277.

\bibitem{PT89}
S. Poljak and Z. Tuza,
On the maximum number of qualitatively independent partitions,
\emph{Journal of Combinatorial Theory, Series A}, 
\textbf{51} (1989), 111--116.

\bibitem{RT96}
G. Rote and R. F. Tichy, 
Quasi-Monte-Carlo methods and the dispersion of point sequences,
\emph{Math. Comput. Modelling}
%\emph{Mathematical and Computer Modelling}
\textbf{23} (1996), 9--23. 

\bibitem{SA95}
M. Sharir and P. K. Agarwal,
\emph{Davenport-Schinzel Sequences and Their Geometric Applications},
Cambridge University Press, Cambridge, 1995.

\bibitem{Sp28}
E. Sperner,
Ein Satz \"uber Untermengen einer endlichen Menge,
\emph{Mathematische Zeitschrift}
\textbf{27} (1928), 544--548.

\bibitem{Ya54}
K. Yamamoto,
Logarithmic order of free distributive lattice,
\emph{Journal of the Mathematical Society of Japan}
\textbf{6} (1954), 343--353.

\end{thebibliography}
\end{document}